\documentclass[11pt]{amsart}
\usepackage{amssymb}
\usepackage{amscd}
\usepackage{color}
\usepackage{comment}
\usepackage[dvipsnames]{xcolor}
\usepackage{url}
\usepackage{hyperref}
\usepackage{enumitem}

\newtheorem{theorem}{Theorem}[section]

\newtheorem{proposition}[theorem]{Proposition}
\newtheorem{corollary}[theorem]{Corollary}

\theoremstyle{definition}
\newtheorem{definition}[theorem]{Definition}

{\par \hfill \fbox{}}

\newtheorem*{theorem*}{Theorem}

\theoremstyle{remark}
\newtheorem{remark}[theorem]{Remark}

\numberwithin{equation}{section}

\newcommand{\norm}[1]{\left\lvert\left\lvert#1\right\rvert\right\rvert}

\DeclareMathOperator*{\Span}{span \;}
\DeclareMathOperator*{\Rpart}{Re \,}
\DeclareMathOperator*{\Impart}{Im}

\newcommand{\Dom}{\textnormal{Dom}}

\newcommand{\EL}{\mathcal{L}}

\newcommand{\C}{\mathbb{C}}

\newcommand{\D}{\mathbb{D}}
\newcommand{\T}{\mathbb{T}}
\newcommand{\PR}{\textnormal{Re}}

\newcommand{\al}{\alpha}
\newcommand{\hol}{\mathcal{H}}

\newcommand{\R}{\mathbb{R}}

\newcommand{\Fi}{\varphi}

\newcommand{\X}{\mathcal{X}}
\newcommand{\HH}{\mathcal{H}}

\newcommand{\inte}{\textnormal{int}}

\begin{document}
	
	\title[Local spectral theory for subordinated operators]{Local spectral theory for subordinated operators: \\the Ces\`aro operator and beyond}

	%   Information for first author
	\author{Eva A. Gallardo-Guti\'{e}rrez}
	\address{Eva A. Gallardo-Guti\'errez \newline
		Departamento de An\'alisis Matem\'atico y Matem\'atica Aplicada,\newline
		Facultad de Ciencias Matem\'aticas,
		\newline Universidad Complutense de
		Madrid, \newline
		Plaza de Ciencias N$^{\underbar{\Tiny o}}$ 3, 28040 Madrid,  Spain
		\newline
		and Instituto de Ciencias Matem\'aticas ICMAT (CSIC-UAM-UC3M-UCM),
		\newline Madrid,  Spain } \email{eva.gallardo@mat.ucm.es}

	%   Information for scond author
	\author{F. Javier Gonz\'alez-Doña}
	\address{F. Javier González-Doña \newline
		Departamento de Matemáticas, \newline
		Escuela Politécnica Superior, \newline
		Universidad Carlos III de Madrid, \newline
		Avenida de la Universidad 30, 28911 Leganés, Madrid, Spain
	\newline
		and Instituto de Ciencias Matem\'aticas ICMAT (CSIC-UAM-UC3M-UCM)
	\newline Madrid,  Spain }
	\email{fragonza@math.uc3m.es}

\thanks{Both authors are partially supported by Plan Nacional  I+D grant no. PID2022-137294NB-I00, Spain,
		the Spanish Ministry of Science and Innovation, through the ``Severo Ochoa Programme for Centres of Excellence in R\&D'' (CEX2019-000904-S \& CEX2023-
001347) and from the Spanish National Research Council, through the ``Ayuda extraordinaria a Centros de Excelencia Severo Ochoa'' (20205CEX001).}

%    General info

\date{January 2025, revised July 2025}


\begin{abstract}
We study local spectral properties for subordinated operators arising from $C_0$-semi\-groups. Specifically, if $\mathcal{T}=(T_t)_{t\geq 0}$ is a $C_0$-semigroup acting boundedly on a complex Banach space and
$$\HH_\nu = \int_{0}^{\infty} T_t\; d\nu(t)$$
is the subordinated operator associated to $\mathcal{T}$, where $\nu$ is a sufficiently regular complex Borel measure supported on $[0,\infty)$, it is shown that $\HH_\nu$ does not enjoy the \emph{Single Valued Extension Property} (SVEP) and has dense \emph{glocal spectral subspaces} in terms of the spectrum of the generator of $\mathcal{T}$.  Likewise, the adjoint $\HH_\nu^{\ast}$ has trivial spectral subspaces and enjoys the Dunford property. As an application, for the classical Ces\`aro operator $\mathcal{C}$
acting  on the Hardy spaces $H^p$ ($1<p<\infty$),  it follows that the local spectrum of $\mathcal{C}$ at any non-zero $H^p$-function or the spectrum of the restriction of  $\mathcal{C}$ to any of its non-trivial closed invariant subspaces coincides with the spectrum of $\mathcal{C}$.
Finally, we characterize the local spectral properties of subordinated operators arising from hyperbolic semigroups of composition operators acting on $H^p$ ($1<p<\infty$), which will depend only on the geometry of the associated Koenigs domain.
\end{abstract}

\maketitle

\section{Introduction}

Our starting point in this work is the remarkable theorem of Kriete and Trutt \cite{KT} which states that the classical Ces\`aro operator $\mathcal{C}$ is \emph{subnormal} in the Hardy space $H^2$, as well as its counterpart in the Banach space setting $H^p$ by Miller, Miller and Smith \cite{MMS} for $1<p<\infty$ and Persson \cite{Per} for $p=1$, showing that $\mathcal{C}$ is \emph{subdecomposable}. More precisely, in $H^2$ the Ces\`aro operator $\mathcal{C}$  is similar to the restriction of a normal operator to one of its closed invariant subspaces and, analogously, in $H^p$ it is similar to the restriction of a decomposable operator. The decomposable operators were introduced by Foia\c{s} \cite{FOIAS} in the sixties as a generalization of spectral operators in the sense of Dunford and despite the fact that many spectral operators in Hilbert spaces as unitary operators, self-adjoint operators or more generally, normal operators are decomposable,  there are still many questions about their restrictions to closed invariant subspaces  unsettled. We refer to the monograph of Laursen and Neumann  \cite{LN00} for more on the subject.

A classical result of Hardy  \cite{Hardy} yields that the Ces\`aro operator $\mathcal{C}$, which takes a complex sequence $\textbf{a}=(a_0, a_1, a_2\dots )$ to that with $n$-th entry:
$$
(\mathcal{C}\, \textbf{a})_n= \frac{1}{n+1} \sum_{k=0}^n a_k, \qquad (n\geq 0),
$$
is bounded in $\ell^2$. Upon identifying sequences with Taylor coefﬁcients of power series, $\mathcal{C}$ acts \emph{formally} on  $f(z)=\sum_{k=0}^{\infty} a_k z^k$ as
\begin{equation}\label{definition Cesaro}
\mathcal{C}(f)(z)=\sum_{n=0}^\infty \left (\frac{1}{n+1} \sum_{k=0}^n a_k \right ) z^n,
\end{equation}
and turns out to be a bounded operator on the classical Hardy spaces on $H^p$ for $1\leq p <\infty$ (see \cite{Siskakis1} and \cite{Siskakis2}, for instance).

\medskip

In this setting, the theory of $C_0$-semigroups turned out to be a powerful tool for studying the Cesàro operator, as shown in Siskakis's paper \cite{Siskakis1} or Cowen's \cite{Cowen}, which gives an alternative proof that $\mathcal{C}$ is subnormal. More recently, in \cite{GP} the authors made use that the adjoint of the Cesàro operator $\mathcal{C}$ appears as the resolvent operator of the infinitesimal generator of the semigroup of composition operators $(C_{\Fi_t})_{t\geq 0}$ acting on $H^p,$ given by $$\Fi_t(z) = e^{-t}z+(1-e^{-t}), \qquad (z\in \D, \ t\geq 0).$$
That is,
$$
(\mathcal{C}^*f)(z) =  -\int_0^\infty e^{-t} f(e^{-t}z+(1-e^{-t}))dt,
$$
for $f \in H^p$ and $z \in \D.$  This representation of $\mathcal{C}^*$ turns out to be a particular instance of what is called a \textit{subordinated operator} associated to a semigroup. In general, if $\mathcal{T}=(T_t)_{t\geq 0}$ is a $C_0$-semigroup of bounded operators acting on a Banach space $X$, and $\nu$ is a `sufficiently regular' complex Borel measure supported on $[0,\infty)$, the operator
$$\HH_\nu = \int_0^\infty T_t \; d\nu(t)$$
is said to be a subordinated operator to $\mathcal{T}$ (see \cite[Remark 3.3.3]{Haase}, for instance).

\medskip

Recently, the authors of \cite{GGM} undertook a study of the Local Spectral Properties in $H^p$-spaces of composition operators induced by linear fractional maps, characterizing in particular the associated local spectra and the local spectral subspaces, among others (see Section \ref{seccion 2} or \cite{LN00} for definitions). Our main goal in this work is to study the local spectral properties of subordinated operators associated to $C_0$-semigroups. Having in mind that the spectrum of the Cesàro operator $\mathcal{C}$ in $H^p$ is the closed disc
$$\sigma(\mathcal{C}; H^p) = \left \{z\in \mathbb{C}:\, \left |z-\frac{2}{p} \right | \leq \frac{2}{p} \right \},$$
(see \cite{Siskakis1}), as a particular instance, we derive the following result for the local spectra and the local spectral subspaces of $\mathcal{C}$:

\begin{theorem*}
\emph{Let $1<p< \infty$ and $\mathcal{C}$ be  the Cesàro operator acting on the Hardy space $H^p.$ Then:}
	\begin{enumerate}
		\item \emph{ The local spectrum of $\mathcal{C}$ at any non-zero $H^p$-function coincides with the spectrum of $\mathcal{C}$.}
		\item \emph{ The local spectral subspace associated to any closed proper subset $F$ of $\sigma(\mathcal{C}; H^p)$ is trivial, namely, it is the zero subspace.}		
         \item \emph{If $M$ is a non-trivial closed invariant subspace of $\mathcal{C}$, then $\sigma(\mathcal{C}\mid_M; H^p)= \sigma(\mathcal{C}; H^p).$}
	\end{enumerate}
%As a particular instance, $\sigma_p(\mathcal{C})=\emptyset$ and $\mathcal{C}$ enjoys the Dunford property $(C).$}
\end{theorem*}

It is worthy to point out the increasing interest in studying (local) spectral properties of the so-called \textit{generalized Cesàro operators}
$$f\in H^p \mapsto \left \{ \begin{array}{ll}
\displaystyle \frac{1}{z} \int_0^z f(s)g'(s)ds, & z\in \mathbb{D}\setminus\{ 0\}, \\
\noalign{\medskip}
f(0)&  z=0;
\end{array} \right.,$$
or the non-normalized version, namely \textit{integration operators}
$$f\in H^p \mapsto \int_{0}^z f(s)g'(s)ds, \qquad (z\in \mathbb{D}),$$
where $g$ is a fixed analytic function in the unit disc.  While these operators turn out to be bounded on $H^p$ whenever $g' \in BMOA$ (see \cite{Aleman}, for instance),  regarding spectral properties or subdecomposability, the  most satisfying results hold assuming that $g'$ is a `small' perturbation of a rational function, and the general questions remain open (see \cite{ AP, APer, LM,Young}, for instance).

\medskip

Subordinated operators in $H^p$  associated to  semigroups (or semiflows) of holomorphic functions  can also be understood as generalizations of Cesàro operators. In this sense, the stated theorem for the Cesàro operator will hold equivalently for a wider class of operators that arise as the adjoints of operators of the form
$$ (\HH_{\nu_\rho}f)(z) = \int_z^1 h'(s) \rho(h(s)-h(z)) f(s) ds, \qquad (z\in \mathbb{D} \text{ and } f\in H^p),$$
where $d\nu_\rho(t) = \rho(t)dt$ is an absolutely continuous measure respect to the Lebesgue measure with  $\rho:[0,\infty)\rightarrow \C$  a sufficiently regular density function and $h$ is the Koenigs function associated to the semigroup (or semiflow) of holomorphic functions $(\Fi_t)_{t\geq 0}$ in $\D$ (see Subsection \ref{subseccion}).

\medskip

We conclude this introduction with an outline of the rest of the manuscript. In Section \ref{seccion 2} we recall some of the preliminaries needed throughout the rest of the text, such as $C_0$-semigroups and their subordinated operators, as well as some background on Local Spectral Theory.
\medskip

In Section \ref{seccion 3} we prove Theorem \ref{main result}, which will allow us to transfer some of the (local) spectral properties of the generator of a $C_0$-semigroup to associated subordinated operators. Indeed, if a weak spectral mapping theorem holds for the $C_0$-semigroup, such properties can be also transferred to the operators of the semigroup (see Remark \ref{remark semigrupo}).

\medskip

In Section \ref{seccion 4} we introduce some of the basic notions of semigroups (or semiflows) of holomorphic functions in $\D$ and derive a result that assures the density of the eigenfunctions of infinitesimal generators of hyperbolic semigroups in terms of the geometric properties of the associated Koenigs domain (Proposition \ref{proposicion densidad}) based on a recent work of Bracci, Gallardo-Guti\'errez  and Yakubovich \cite{BGY}. As a consequence, we prove a result that provides local spectral properties for subordinated operators  depending only on the geometry of the Koenigs domain associated to the hyperbolic semigroup (Theorem \ref{teorema semigrupos}). Finally, we exhibit particular instances of subordinated operators for which the result applies (Subsection \ref{subseccion}) and discuss further applications to  non-hyperbolic semigroups of composition operators.

\section{Preliminaries}\label{seccion 2}

Throughout the rest of the text, $X$ will stand for an infinite-dimensional complex Banach space and $\EL(X)$ for the algebra of linear bounded operators acting on $X$.

Recall that a one-parameter family $\mathcal{T}=(T_t)_{t\geq 0} \subset \EL(X)$ is a \emph{$C_0$-semigroup} if:
\begin{enumerate}
	\item $T_0 = Id_X,$ the identity operator;
	\item $T_{t+s}= T_t\, T_s$ for every $t,\, s \geq 0$;
	\item $T_t \rightarrow Id_X$ when $t\rightarrow 0$ in the strong operator topology.
\end{enumerate}
Every $C_0$-semigroup satisfies
\begin{equation}\label{exponential bound}
\norm{T_t} \leq M e^{\omega t} \qquad \textnormal{for } t\geq 0,
\end{equation}
for some $ M\geq0$ and $\omega \in \R$. The infimum of those $\omega$ such that \eqref{exponential bound} holds for some $M\geq 0$ is called the \emph{type} of the $C_0$-semigroup (see \cite[p. 250]{EN})and can be computed as
$$\omega_0 := \inf\limits_{t>0} \frac{\log (\norm{T_t})}{t}.$$

Given a $C_0$-semigroup $(T_t)_{t\geq0}$ on a Banach space $X $, recall that its generator is the  closed and
densely defined linear operator $\Delta$ defined by
\[
\Delta x=\lim_{t\to 0^+} \frac {T_t x-x} t
\]
with domain $\textnormal{Dom}(\Delta)=\{x\in X : \lim_{t\to 0^+} \frac {T_t x-x} t \enspace \text{exists}\}$. The semigroup is
determined uniquely by its generator.

\medskip

It is a classical fact that the spectrum of $\Delta$ in $X$, namely, those complex numbers $z$ such that $\Delta-zI$ is not invertible in $X$, satisfies
$$\sigma(\Delta) \subseteq \{z\in \C: \PR(z)\leq \omega_0 \}$$
(see \cite[Proposition 2.2]{EN}, for instance).

\bigskip

\noindent \emph{A word about notation.}
For simplicity, the spectrum of an operator $T$ on $X$ is denoted by $\sigma(T)$ unless it is necessary to specify the underlying space, when it is denoted by $\sigma(T; X)$.

\bigskip

For $z \in \C$ with $\PR(z) > \omega_0$, it is possible to express the resolvent operator
$$R(z,\Delta)= -(\Delta-z I)^{-1}$$ in terms of the Laplace transform of $(T_t)_{t\geq 0},$  namely
\begin{equation}\label{laplace transform}
R(z,\Delta)x = \int_{0}^\infty e^{-z t}\, T_t x \, dt \qquad (x\in X),
\end{equation}
(see \cite[p. 231]{EN}). Likewise, the spectral properties of the resolvent operators $R(z,\Delta)$ as well as the operators $T_t$ in the semigroup can be related with the ones of the infinitesimal generator by means of \emph{spectral mapping theorems}. In particular, if $\sigma_p (T)$ denotes the point spectrum of an operator $T$, namely, the set consisting of eigenvalues of $T$, the following holds:
\begin{proposition}\label{spectral mapping resolvente}
Let $X$ be an infinite-dimensional complex Banach space, $(T_t)_{t\geq 0}\subset \EL(X)$ a $C_0$-semigroup and $\Delta$ its infinitesimal generator. Then,
\begin{enumerate}
\item[\rm{(1)}] $\{e^{zt}:z\in \sigma(\Delta)\} \subseteq \sigma(T_t)$ for every $t\geq 0$;
\item[\rm{(2)}] $\{e^{zt}:z\in \sigma_p(\Delta)\} = \sigma_p(T_t)\setminus\{0\}$ for every $t\geq 0$.
\end{enumerate}
Moreover, for every complex number $\lambda$ in the resolvent of $\Delta$, it holds that
\begin{enumerate}
\item[\rm{(3)}] $\sigma(R(\lambda,\Delta)) = \{0\}\cup \{ (\lambda-z)^{-1} : z \in \sigma(\Delta) \}$;
\item[\rm{(4)}] $\{ (\lambda-z)^{-1} : z \in \sigma_p(\Delta) \}\subseteq \sigma_p(R(\lambda,\Delta)).$
\end{enumerate}
\end{proposition}
\noindent We refer to  \cite[Chapter IV, Sections 3.6 and 3.7]{EN} and \cite[Theorem 6.4 and Corollary 6.6]{Haase1}, for instance.

\medskip

The equation \eqref{laplace transform} provides particular instances of subordinated operators associated to a semigroup according to the following:

\begin{definition}
Let $\nu$ be a finite complex Borel measure in $[0,+\infty)$ such that $\int_{0}^\infty e^{(\omega_0+\delta)t}\, d|\nu|(t) < \infty$ for some $\delta >0$ and
$\mathcal{T}=(T_t)_{t\geq 0}\subset \EL(X)$ a $C_0$-semigroup. The subordinated operator $\HH_\nu$ associated to $\mathcal{T}$ is given by
\begin{equation}\label{def sub}
\HH_\nu x = \int_0^\infty \, T_t  x \ d\nu(t), \qquad (x \in X).
\end{equation}
\end{definition}

\smallskip

Before going further, let us point out that \eqref{def sub} defines a linear bounded operator in $X$ because the integral is Bochner-convergent. Clearly, the absolute continuous measure $d\nu_{\lambda}(t) = e^{-\lambda t}dt$ where $\lambda \in \mathbb{C}$ with $\PR(\lambda) > \omega_0$ yields that $\HH_{\nu_{\lambda}}=R(\lambda,\Delta).$

\medskip

\subsection{Sectorial operators}

In order to study the local spectral properties of subordinated operators, we recall some spectral mapping theorems in the lines of Proposition \ref{spectral mapping resolvente}. For such a purpose, we introduce a \emph{toolkit} regarding  functional calculus for sectorial operators and refer the reader to the monograph \cite{Haase}.
% (see also Section 5 in the recent paper by Abadías and Oliva  \cite{AO}, where the spectrum, point spectrum and essential spectrum of certain subordinated operators arising from weighted composition operators semigroups is studied).

\medskip

Given $\nu$  a finite complex Borel measure in $[0,\infty)$ of bounded variation, the \emph{Laplace transform} of $\nu$ is given by
$$\EL(\nu)(z)= \int_0^\infty e^{-z t}\, d\nu(t),$$
where $z \in \C$ with $\PR(z) \geq 0$. It is a classical fact that $\EL(\nu)$ is a holomorphic function in the right-half plane $\{z\in \C : \PR(z) > 0\}$ and continuous in its closure.

\medskip

Given $0\leq \omega \leq \pi$, let $S_\omega$ be the sector in the complex plane given by
$$
S_\omega = \left\{ \begin{matrix}
&\{z\in \C\setminus \{0\}:\, \ |\arg z| < \omega \} & \textnormal{if} & \omega \in (0,\pi]  \\
&\{z\in \C\setminus \{0\}:\, \ \Impart(z)=0, \Rpart(z)>0 \} & \textnormal{if} & \omega = 0.
\end{matrix} \right.$$

Recall that a (possibly unbounded operator) $\Delta$ acting on $X$ is called \emph{sectorial of angle $\omega \in [0,\pi)$}, denoted by $\Delta \in \textnormal{Sect}(\omega)$ if it satisfies
\begin{enumerate}
	\item [(i)] $\sigma(\Delta)\subset \overline{S_\omega}$ and
	\item [(ii)] $M(A,\omega') := \sup \{\norm{\lambda\,  R(\lambda,\Delta)}: \lambda \in \C\setminus \overline{S_{\omega'}} \}  <\infty$ for all $\omega'\in (\omega,\pi)$.
\end{enumerate}

It is a standard fact that if $\Delta$ is the infinitesimal generator of a uniformly bounded semigroup of operators, that is,
$$\sup\limits_{t\geq 0}\norm{T_t}<\infty,$$
then $-\Delta$ is a sectorial operator of angle $\pi/2.$ In particular, if $(T_t)_{t\geq 0} \subset \EL(X)$ is a $C_0$-semigroup of type $\omega_0 \in \R$ and $\Delta$ is its infinitesimal generator, the semigroup
$$(e^{-(\omega_0+\varepsilon)t}\, T_t)_{t\geq 0}$$
is a uniformly bounded semigroup for every $\varepsilon>0$, so $(\omega_0+\varepsilon)I-\Delta$ is sectorial of angle $\pi/2.$

\medskip

Now, we introduce a functional calculus for sectorial operators of angle $\pi/2$. First, we recall the following:

\begin{definition}
Let $A$ be  a sectorial operator of angle $\pi/2$ in $X$ and $\omega \in (\pi/2,\pi)$. The domain of the functional calculus of $A$, denoted by $\mathcal{E}(A)$, consists of  holomorphic functions $f: S_\omega \rightarrow \C$  such that the limits
$$ \lim\limits_{z\rightarrow 0} f(z) :=d_0 \qquad \lim\limits_{z\rightarrow \infty} f(z) :=d_\infty$$
are finite and satisfy that for some $r,R>0$, that the integrals
$$ \int_{\partial(S_{\omega'}\cap\{|z|<r\})} \left| \frac{f(z)-d_0}{z}\right| |dz| <\infty, \qquad \int_{\partial(S_{\omega'}\cap\{|z|>R\})} \left| \frac{f(z)-d_\infty}{z}\right| |dz| <\infty $$ converges for every $\omega'\in (\pi/2, \omega).$
\end{definition}
Here $\partial(S_{\omega'}\cap\{|z|<r\})$ denotes the boundary of the set $S_{\omega'}\cap\{|z|<r\}$.

\smallskip

Given $A$  a sectorial operator of angle $\pi/2$ in $X$ and $f\in \mathcal{E}(A)$, the operator $f(A)$ may be defined as
$$f(A)= d_\infty + d_0(I+A)^{-1} + \int_{\partial S_{\omega'}} (f(z)-d_\infty-\frac{d_0}{z+1})(z I-A)^{-1}dz,$$
where $\omega' \in (\pi/2,\pi).$ Note that this definition does not depend on the selected $\omega'$.

\medskip

In particular, if $\mathcal{T}=(T_t)_{t\geq 0}$ is a $C_0$-semigroup with infinitesimal generator $\Delta$ of type $w_0\in \R$, the operator $(\omega_0+\varepsilon)I-\Delta$ is a sectorial operator of angle $\pi/2$, so the functional calculus defined applies.

In order to simplify the notation, and following that of \cite[p. 27]{AO}, we write
$f \in \mathcal{E}(-\Delta)$ if $f_{\omega_0+\varepsilon} \in \mathcal{E}((\omega_0+\varepsilon)I-\Delta),$ where
$f_{\omega_0+\varepsilon}(z) = f(z-\omega_0-\varepsilon)$. In such a case, we define
$$f(-\Delta)= f_{\omega_0+\varepsilon}((\omega_0+\varepsilon)I-\Delta).$$

\medskip

With this notation, we can state the desired Spectral Mapping Theorem for our purposes:

\begin{theorem}\label{teorema spectral mapping}
Let $\mathcal{T}=(T_t)_{t\geq 0} \subset \EL(X)$ be a $C_0$-semigroup of operators with infinitesimal generator $\Delta$ of type $w_0\in \R$.  Let $\nu$ be a finite complex Borel measure in $[0,\infty)$ such that $\int_{0}^\infty e^{(\omega_0+\delta)t} \, d|\nu|(t) < \infty$ for some $\delta > 0.$
Let $\HH_{\nu}$ be the subordinated operator associated to $\mathcal{T}$ and assume that $\EL(\nu) \in \mathcal{E}(-\Delta).$ Then the following assertions hold:
\begin{enumerate}
\item[\rm{(1)}] The spectrum of $\HH_{\nu}$ is
\begin{equation*}\label{contenciones espectrales1}
\sigma(\HH_\nu) = \EL(\nu)(\tilde{\sigma}(-\Delta))
\end{equation*}
where $\tilde{\sigma}(-\Delta) = \sigma(-\Delta)\cup\{\infty\}$.
\item[\rm{(2)}] The point spectrum of $\HH_{\nu}$ satisfies
\begin{equation*}\label{contenciones espectrales2}
\EL(\nu)(\sigma_p(-\Delta))\subseteq \sigma_p(\HH_\nu) \subseteq \EL(\nu)(\sigma_p(-\Delta))\cup \EL(\nu)(M),
\end{equation*}
where $M = \{0,\infty\}\cap \tilde{\sigma}(-\Delta).$
\end{enumerate}
Moreover, $\ker(-\Delta-\lambda I) \subseteq \ker(\HH_\nu -\EL(\nu)(\lambda)\,I)$ for every $\lambda \in \sigma_p(-\Delta).$
\end{theorem}

\begin{proof}
An argument along the lines of  \cite[Corollary 5.3]{AO} yields that $\EL(\nu)\in \mathcal{E}(-\Delta).$ Moreover, $\HH_\nu = \EL(\nu)(-\Delta)$ by \cite[Proposition 3.3.2]{Haase},  so both statements (1) and (2) in Theorem \ref{teorema spectral mapping}  follows by \cite[Theorem 6.4 and Corollary 6.6]{Haase1}.

Finally, let $\lambda \in \sigma_p(-\Delta)$ and $x \in \ker(-\Delta-\lambda I).$ By Proposition \ref{spectral mapping resolvente}, $x \in \ker(T_t-e^{-\lambda t}).$ Since, $$\HH_\nu x = \int_0^\infty \, T_t x \, d\nu(t) = x \int_0^\infty e^{-\lambda t}d\nu(t) = \EL(\nu)x,$$ the last statement of Theorem \ref{teorema spectral mapping} holds and the proof follows.
\end{proof}

%Abadías and Oliva \cite{AO} , in order to relate the spectrum, point spectrum and essential spectrum of certain subordinated operators arising from weighted composition operators to the spectral sets of the generator of the semigroup.

It is worthy to remark that as a consequence of \cite[Lemma 5.2 and Corollary 5.3]{AO}, next proposition yields a relatively easy way to provide measures $\nu$ satisfying the hypothesis of the previous result.

\begin{proposition} Let $\nu$ be a finite complex Borel measure on $[0,\infty)$ such that $\int_{0}^\infty e^{(\omega_0+\delta)t} |d\nu|(t) < \infty$ for some $\delta > 0.$ Assume $d\nu(t) = h(t)dt,$ where $h$ can be extended holomorphically to a sector $S_{\omega}$ with $\omega \in (0,\pi/2]$. Suppose further that there exist $\eta \in (0,1]$ and $\xi\in(0,1)$ such that
$$
\sup_{z\in S_\xi\cap\{|z|\leq 1\}} |z^{1-\eta}h(z)| < \infty \qquad \sup_{z\in S_\xi\cap\{|z|\geq 1\}} |z^{1+\xi}e^{(\omega_0+\delta)z}h(z)| < \infty
$$ for all $0<\varepsilon<\omega.$  Then, $\EL(\nu) \in \mathcal{E}(-\Delta)$.
\end{proposition}
	
We conclude this subsection with a basic fact about adjoints of $C_0$-semigroups of operators. For such a task, recall that if  \( Z \) is a closed operator with dense domain \( \mathcal{D} \) in a Banach space $X$, the \emph{adjoint} \( Z^* \) with domain \( \mathcal{D}^* \subseteq X^* \) is defined as follows:
\( \phi \in \mathcal{D}^* \) if and only if the linear functional \( f \mapsto \langle Zf, \phi \rangle \) is norm bounded on \( \mathcal{D} \);
if \( \phi \in \mathcal{D}^* \), then \( Z^* \phi \) is the unique bounded extension of that linear functional to $X$.
Hence
\[
\langle Zf, \phi \rangle = \langle f, Z^* \phi \rangle
\]
for all \( f \in \mathcal{D} \) and \( \phi \in \mathcal{D}^* \) (see \cite{Davies}).

\begin{theorem}\cite[Theorem 1.34]{Davies}
Let $X$ be a reflexive Banach space  and $(T_t)_{t\geq 0} \subseteq \EL(X)$ a $C_0$-semigroup with infinitesimal generator $\Delta.$ Then, $(T_t^*)_{t\geq 0}\subset \EL(X^*)$ is also a $C_0-$semigroup whose infinitesimal generator is $\Delta^*.$
\end{theorem}

\subsection{Local Spectral Theory} In this subsection we will recall some of the basic features of Local Spectral Theory and refer the reader to the monograph by Laursen and Neumann \cite{LN00}.

\medskip

Given $T\in \EL(X)$ and $x \in X$, the \emph{local resolvent} of $T$ at $x$, denoted as $\rho_T(x)$, is defined as the union of all open sets $U\subset \C$ such that there exists an analytic function $f_x : U \rightarrow X$ satisfying
\begin{equation}\label{local resolvent}
(T-zI)f_x(z) = x \qquad (z\in U).
\end{equation}
In particular, $\rho_T(x)$ is an open set that contains the resolvent of the operator $\rho(T)$. An operator $T\in \EL(X)$ is said to have the \emph{single-valued extension property} (SVEP) if for every open set $U$ and every analytic function $f_x:U\rightarrow X$ the equality
$$(T-zI)f_x(z) = 0 \qquad (z\in U)$$
implies that $f_x\equiv 0.$ This property assures that the analytic function solving \eqref{local resolvent} is unique, so one can define the \emph{local resolvent function} of $T$ at $x$ as the unique analytic map $f_x:\rho_T(x)\rightarrow X$ such that \eqref{local resolvent} holds.

\medskip

The \emph{local spectrum} of $T\in \EL(X)$ at $x\in X$ is defined as $\sigma_T(x)= \C\setminus \rho_T(x).$ This set is always a (possibly empty) compact subset of the spectrum $\sigma(T)$. Having in mind the spectral radius formula
$$r(T) = \limsup\limits_{n\rightarrow \infty} \norm{T^n}^{1/n},$$
the \textit{local spectral radius} of $T$ at $x$ is defined by
$$r_T(x) := \limsup\limits_{n\rightarrow\infty} \norm{T^nx}^{1/n}.$$
If $T$ has the SVEP, a local spectral radius formula satisfies:
\begin{equation}\label{local spectral radius formula}
r_T(x) = \max\{|\lambda| : \lambda \in \sigma_T(x) \}, \qquad (x\in X\setminus\{0\}),
\end{equation}
(see \cite[Proposition 3.3.13]{LN00}).

The definition of the local spectrum allows us to introduce the local spectral subspaces, or simply, the spectral subspaces associated to $T$.

For any set $A\subset \C$, the \emph{spectral subspace} of $T$ associated to $A$, denoted as $X_T(A)$, is defined as
$$X_T(A)= \{x \in X : \sigma_T(x)\subseteq A\}.$$
In other words, $X_T(A)$ consists of those  vectors $x\in X$ for which $(T-zI)$ is locally invertible for $z\in \C\setminus A.$ We note that $X_T(A) = X_T(A\cap\sigma(T))$  for every $A\subset \C.$

The spectral subspaces turn out to be (non-necessarily closed) subspaces which are hyperinvariant by $T$, namely, invariant under every operator in the commutant of $T$. If all the spectral subspaces of an operator $T$ associated to closed sets are closed, then $T$ is said to enjoy the \emph{Dunford property (C)}. Clearly, the Dunford property implies the SVEP.

\medskip

We also recall the \emph{glocal spectral subspaces} $\X_T(F)$, which are a variant of spectral subspaces that in general are better suited for operators lacking the SVEP. If $F\subset \C$ is a closed set, the glocal spectral subspace $\X_T(F)$ consists of all $x\in X$ such that there exists an analytic function $f_x:\C\setminus F \rightarrow X$ satisfying
$$(T-zI)f_x(z) = x, \qquad (z\in \C\setminus F).$$
If $T$ has the SVEP, then $X_T(F) = \X_T(F)$ for every closed set $F\subset \C$. In general, one has the inclusion $\X_T(F)\subseteq X_T(F)$.

\medskip

Glocal spectral subspaces are also of interest when dealing with adjoint of operators. Namely, if $X$ is a complex Banach space and $T\in \EL(X)$,  for every disjoint closed sets $F_1,F_2 \subseteq \C$, it follows:
\begin{equation}\label{contenciones espectrales}
 \X^*_{T^*}(F_1)\subseteq \X_T(F_2)^\perp,
\end{equation}
 (see \cite[Proposition 2.5.1]{LN00}).

\medskip

\section{Transferring the (local) spectral properties of infinitesimal generators}\label{seccion 3}

In this section we prove Theorem \ref{main result} which  will allow us to transfer some of the (local) spectral properties of the infinitesimal generator $\Delta$ of a $C_0$-semigroup $\mathcal{T}=(T_t)_{t\geq 0}$ to subordinated operators $\HH_\nu$ associated to $\mathcal{T}$, as well as for their adjoints.
The result extends the conclusions of Theorem 2.1 in \cite{GGM} to the context of a $C_0$-semigroups and reads as follows:

\begin{theorem}\label{main result}
Let $X$ be a complex Banach space and $\mathcal{T}=(T_t)_{t\geq 0} \subset \EL(X)$ a $C_0$-semigroup of type $w_0\in \R$  and infinitesimal generator $\Delta$. Assume that $\overline{\inte(\sigma(-\Delta))} = \sigma(-\Delta)$ is connected and that there exists a non-vanishing analytic map $f_\Delta : \inte(\sigma(-\Delta)) \rightarrow \textnormal{Dom}(-\Delta)$ such that
$$(-\Delta-zI)f_\Delta(z) = 0 \qquad \text{ for } z\in \inte(\sigma(-\Delta)).$$
Assume further that the linear span
$$\Span\{ f_\Delta(\lambda) : \lambda \in \inte(\sigma(-\Delta)) \}$$
is dense in $X$. Let $\nu$ be a finite complex Borel measure supported in $[0,\infty)$ such that $\int_{0}^\infty e^{(\omega_0+\delta)t}d|\nu|(t)<\infty$ for some $\delta > 0$ and $\EL(\nu)$ is non-constant in  $\mathcal{E}(-\Delta)$. Then the subordinated operator $\HH_\nu$ to $\mathcal{T}$ satisfies the following properties:

\begin{enumerate}
\item $\HH_\nu$ does not enjoy the SVEP.
\item  For every relatively open subset non-empty $U\subset \sigma(\HH_\nu),$ $\X_{\HH_\nu}(\overline{U})$ is dense in $X$.
\item $\sigma_p(\HH_\nu^*) = \emptyset.$
\item  $X_{\HH_\nu^*}(F) = \{0\}$ for every closed subset $F\subsetneq \sigma(\HH_\nu^*).$
\item  $\HH_\nu^*$ has the Dunford property (C).
\item  $\sigma_{\HH_\nu^*}(x) = \sigma(\HH_\nu^*)$ for every $x\in X^*\setminus \{0\}.$
\item  $r_{\HH_\nu^*}(x) = r(\HH_\nu^*)$ for every $x\in \X^*\setminus \{0\}.$
\item  If $M$ is a non-trivial closed invariant subspace for $\HH_\nu^*$, then
$ \sigma(\HH_\nu^*) \subseteq \sigma(\HH_\nu^*\mid_M)\subset \eta(\sigma(\HH_\nu^*)),$ where $\eta(\sigma(\HH_\nu^*))$ denotes the full spectrum of $\sigma(\HH_\nu^*)$, namely, the union of $\sigma(\HH_\nu^*)$ and all the bounded components of the resolvent set $\rho(\HH_\nu^*).$
\end{enumerate}
\end{theorem}

\begin{proof}
We begin by noting  that the existence of the analytic function $f_\Delta$ implies that $\inte(\sigma(-\Delta))\subseteq \sigma_p(-\Delta).$
%Let $\nu$ be a complex Borel measure satisfying the hypothesis of the result, and let $\HH_\nu \in \EL(X)$.
In particular, by Theorem \ref{teorema spectral mapping}, it follows that
$$\sigma(\HH_\nu) = \EL(\nu)(\tilde{\sigma}(-\Delta)) \text{ and } \EL(\nu)(\sigma_p(-\Delta))\subseteq \sigma_p(\HH_\nu).$$

Let us start with the proof of $(1)$. Since $\EL(\nu)$ is holomorphic in the translated sector $-(\omega_0+\delta)+S_{\omega}$, which contains $\sigma(-\Delta),$ the set  $\EL(\nu)(\inte(\sigma(-\Delta)))$ is open and
\begin{equation} \label{spectrum containment}
\EL(\nu)(\inte(\sigma(-\Delta)))\subset \sigma(\HH_\nu).
\end{equation}
By hypothesis, the function $\EL(\nu)'$ is not zero,  so there exists an open subset $G\subseteq \inte(\sigma(-\Delta))\subseteq \sigma_p(-\Delta)$ such that
$$\EL(\nu) : G \rightarrow \EL(\nu)(G)$$
is biholomorphic. The analytic map $f_\Delta \circ \EL(\nu)^{-1} : \EL(\nu)(G) \rightarrow X$ satisfies that
$$(\HH_\nu-zI)f_\Delta\circ \EL(\nu)^{-1}(z) = 0$$
for every $z\in \EL(\nu)(G)$. Accordingly, $\HH_\nu$ does not enjoy the SVEP and $(1)$ follows.

\medskip

In order to prove $(2)$, we note that for every non-empty open set $\Omega \subset \inte(\sigma(-\Delta))$,
$$\Span \{f_{\Delta}(\lambda): \lambda \in \Omega\}$$
is dense in $X$  since the map $f_\Delta : \inte(\sigma(-\Delta)) \rightarrow \textnormal{Dom}(-\Delta)$ is holomorphic. Likewise, the hypotheses on the spectrum $\sigma(-\Delta)$ yield that it has no isolated points.
So, if $U\subset \sigma(\HH_\nu)$ is a relatively open non-empty set, having in mind that $\sigma(\mathcal{H}_\nu) = \mathcal{L}(\nu)(\tilde{\sigma}(-\Delta))$, $\textnormal{int}(\sigma(-\Delta))\subset \sigma_p(-\Delta)$ and $\mathcal{L}(\nu)(\sigma_p(-\Delta))\subseteq \sigma_p(\mathcal{H}_\nu)$, there exists an open subset $U_1 \subseteq U\cap \sigma_p(\mathcal{H}_\nu).$  Now, by Theorem \ref{teorema spectral mapping},  $U_1 \subseteq U\cap\sigma_p(\mathcal{H}_\nu) = U\cap\mathcal{L}(\nu)(\sigma_p(-\Delta))\setminus \mathcal{L}(\nu)(M)$, so there exists a non-empty open subset $U_2\subseteq \mathcal{L}(\nu)(\textnormal{int}(\sigma_p(-\Delta)))\cap U_1$, since $\mathcal{L}(\nu)(M)$ consists, at most, of two points.
Accordingly, there exists a non-empty open subset $G\subseteq \textnormal{int}(\sigma_p(-\Delta)) \subseteq \textnormal{int}(\sigma(-\Delta))$ such that $\mathcal{L}(\nu)(G)$ is contained in $U_2$.

Now,
$$\Span\{f_{\Delta}(\lambda) : \lambda \in G \}$$
is dense in $X$ and observe that
$$\HH_\nu f_{\Delta}(\lambda) = \EL(\nu)(\lambda) f_{\Delta}(\lambda),\qquad (\lambda \in G);$$
so $\EL(\nu)(\lambda)$ is an eigenvalue for $\HH_\nu$ for every $\lambda \in G$. But $\EL(\nu)(G)\subset U,$ so by \cite[Proposition 3.3.1]{LN00} and the inclusion properties of glocal spectral subspaces, it yields that $\Span\{f(\lambda): \lambda \in G\} \subseteq \X_{\HH_\nu}(\overline{U})$ and $(2)$ follows.

\medskip

In order to show $(3)$, recall that $\sigma(-\Delta)$ is not a singleton by hypothesis, so neither is $\sigma(\HH_\nu)$. Fix $\lambda_0 \in \sigma(\HH_\nu^*)=\sigma(\HH_\nu).$ By \cite[Proposition 3.3.1]{LN00}, we have that
$$\ker(\HH_\nu^*-\lambda I) \subset \X^*_{\HH_\nu^*}(\{\lambda_0\}).$$
Now, let $U\subset \sigma(\HH_\nu)$ be a non-empty relatively open set not containing $\lambda_0$ in its closure, which can be achieved since the spectrum $\sigma(\HH_\nu)$ is not a singleton. Then, having in mind \eqref{contenciones espectrales}, we deduce that
$$\X^*_{\HH_\nu^*}(\{\lambda_0 \}) \subseteq \X_{\HH_\nu}(\overline{U})^\perp = \{0\}.$$
Hence, $\lambda_0 \notin \sigma_p(\HH_\nu)$, which concludes the proof. Observe that, in particular, $\HH_\nu^*$ has the SVEP, and $X^*_{\HH_\nu^*}(F) = \X^*_{\HH_\nu^*}(F)$ for every closed set $F\subset \C.$

\medskip

Let us prove $(4)$. Let $F$ be a proper closed subset of $\sigma(\HH^*_\nu).$ Since $F$ is proper, there exists a relatively open subset $U\subset \sigma(\HH_\nu^*)$ such that $\overline{U}\cap F = \emptyset.$  Then, as in the proof of $(3)$, by \eqref{contenciones espectrales} and  the density of $\X_{\HH_\nu}(\overline{U})$, we deduce that $X^*_{\HH_\nu^*}(F) = \{0\}$, as we wished to prove.

\medskip

As a consequence of $(4)$, the proof of $(5)$ follows directly: if $F\subsetneq \sigma(\HH_\nu^*)$, then $X^*_{\HH_\nu^*}(F) = \{0\},$ and if $F=\sigma(T)$, it is trivial to check that $X^*_{\HH_\nu^*}(F)= X^*.$ In particular, $X^*_{\HH_\nu^*}(F)$ is a closed linear manifold for every closed set $F\subset \sigma(\HH_\nu^*)$, which yields the statement.

\medskip

In order to prove $(6)$ and $(7)$, we observe that statement $(4)$ directly implies that $(6)$ holds, and such identity together with \eqref{local spectral radius formula} yields $(7)$.

\medskip

To conclude the proof of Theorem \ref{main result}, we show $(8)$. Let $M\subset X^*$ be a non-trivial closed invariant subspace for $\HH_\nu^*$. Then, by \cite[Proposition 1.2.16 (e)]{LN00} it follows that $M\subset X_{\HH_\nu^*}(\sigma(\HH_\nu^*\mid_M))$. But $M\neq \{0\}$, so by $(4)$ we  deduce that $\sigma(\HH_\nu^*)\subseteq\sigma(\HH_\nu^*\mid_M).$ For the second containment, it is enough to apply \cite[Theorem 0.8]{RR}.
\end{proof}

\begin{remark}\label{remark semigrupo}
Under the hypothesis of Theorem \ref{main result}, if a weak spectral mapping formula of the form
	\begin{equation}\label{weak spectral mapping formula}
\sigma(T_t) = \overline{\{ e^{\lambda t}: \lambda \in \sigma(\Delta) \}} \qquad (t\geq 0)
	\end{equation}
holds, then the statements $(1)-(8)$ in Theorem \ref{main result} are also satisfied for each of the operators $T_t$ in the semigroup $\mathcal{T}$ besides the subordinated operator $\HH_\nu$ associated to it.

\medskip

It is worth mentioning that the equality \eqref{weak spectral mapping formula} does not hold for every $C_0$-semigroup acting on a complex Banach space $X$, see for instance \cite[Chapter IV, Section 3.3]{EN}. Nevertheless, \eqref{weak spectral mapping formula}, even in a stronger version, does hold for a large variety of $C_0$-semigroups, such as eventually compact semigroups, analytic semigroups, uniformly continuous semigroups or $C_0$-semigroups of normal operators on Hilbert spaces \cite[Chapter IV, Corollaries 3.12 and 3.14]{EN}.
\end{remark}

\section{Subordinated operators arising from semigroups of composition operators}\label{seccion 4}

Our main goal in this section is applying Theorem \ref{main result}  to a large collection of semigroups of composition operators acting on the Hardy spaces $H^p$ of the unit disc $\D$, $1<p<\infty$,  which will allow us  to get local spectral properties of the subordinated operators associated to such semigroups as well as of their adjoints. A particular instance is the classical Ces\`aro operator, mentioned in the introduction. The subordinated operators that emerge  turn out to be \textit{averaging operators} introduced by Siskakis in 1993 \cite{Siskakis3}.

\medskip

\subsection{Semigroups of analytic functions in $\D$} Recall that a  family $(\Fi_t)_{t\geq 0}$ of holomorphic functions on $\D$ is said to be a \emph{semigroup (or semiflow)} in $\D$ if the map $t \in [0,+\infty) \mapsto \hol(\D)$ is a continuous semigroup homomorphism, where $[0,+\infty)$ is endowed with the Euclidean topology and $\hol(\D)$ with the topology of uniform convergence on compact subsets of $\D.$ In particular, the following properties yield:
\begin{enumerate}
	\item $\Fi_0(z) = z$ for every $z \in \D.$
	\item $\Fi_t\circ \Fi_s = \Fi_{t+s}$ for every $s,t \geq 0.$
	\item $\lim_{t\to s} \Fi_t = \Fi_s$  uniformly on compact subsets of $\D$ for $s,t \geq 0.$
\end{enumerate}
In such a case, the function $t \mapsto \Fi_t(z)$ is differentiable for each $z \in \D$, and there exists a unique function $G: \D\rightarrow \C$ such that
$$ \frac{\partial \Fi_t(z)}{\partial_t} \mid_{t=t_0} = G(\Fi_{t_0}(z)), \qquad (z\in \D,\,  t_0 \geq 0).$$
This function $G$ is the \textit{infinitesimal generator} of the semigroup $(\Fi_t)_{t\geq 0}.$

\medskip

The semigroups in $\D$ can be classified with respect to the configuration of the fixed points of the functions of the semigroups. In this respect, a semigroup $(\Fi_t)_{t\geq 0}$ is \emph{elliptic} if there exists $z_0 \in \D$ such that $\Fi_{t_0}(z_0) = z_0$ for some $t_0>0.$ Otherwise, we will say that the semigroup is \textit{non}-elliptic.

In this latter case, there exists a unique point $\tau \in \T,$ called the \emph{Denjoy-Wolff point} such that $\Fi_t$ converges uniformly on compact subsets of $\D$ to the constant function $z\mapsto \tau$ when $t\rightarrow + \infty.$ Moreover, $\Fi_t(z)$ has non-tangential limits $e^{t\al}$ at the Denjoy--Wolff point, with some $\al \leq 0.$ If $\al < 0,$ the semigroup $(\Fi_t)_{t\geq 0}$ is said to be \emph{hyperbolic}, and it is said to be \emph{parabolic} when $\al = 0.$

\medskip

For non-elliptic semigroups, there exists a unique representation
\begin{equation}\label{representacion koenigs}
\Fi_t(z) = h^{-1}(h(z)+t), \qquad (z \in \D, \ t \geq 0),
\end{equation}
where $h : \D \mapsto \C$ is a univalent function with $h(0)=0$. $h$ is called the \emph{Koenigs function} of the semigroup $(\Fi_t)_{t\geq 0}$ and it is unique up to the normalization $h(0)=0$. The infinitesimal generator $G$ of the semigroup $(\Fi_t)_{t\geq 0}$ and the Koenigs function are related by means of
\begin{equation}\label{derivada h}
G(z)h'(z) = 1 \qquad (z\in \D).
\end{equation}

The range of the Koenigs function $\Omega = h(\D)$ is the associated \emph{Koenigs domain} of $(\Fi)_{t\geq 0}$, which is  \textit{convex in the positive direction}, that is, if $w\in \Omega$ then $w+t \in \Omega$ for every $t \geq 0.$ The geometry of Koenigs domain plays a significant role in understanding the semigroup. In particular,  a non-elliptic semigroup $(\Fi_t)_{t\geq 0}$ is hyperbolic if and only if its Koenigs domain $\Omega$ is contained in a horizontal strip (we refer to the recent monograph \cite{BCD} for this and further results and references).

\subsection{Semigroups of composition operators on $H^p$}
It is a well known result that every semigroup $(\Fi_t)_{t\geq 0}$ in $\D$ induces a $C_0$-semigroup  of composition  operators $(C_{\Fi_t})_{t\geq 0}$ acting on the Hardy space $H^p$ of the unit disc when $1\leq p < \infty.$ In such a case, the infinitesimal generator $\Delta$ of the semigroup $(C_{\Fi_t})_{t\geq 0}$ is given by
$$\Delta f = G f', \qquad (f \in \Dom(\Delta)),$$ where $G$ is the infinitesimal generator of $(\Fi_t)_{t\geq 0}$ and $\Dom(\Delta) = \{f\in H^p : Gf' \in H^p \}.$

Assume that $(\Fi_t)_{t\geq 0}$ is a non-elliptic semigroup. It is straightforward that $\lambda \in \sigma_p(\Delta)$ if and only if there exists $f\in \Dom(\Delta)$ such that $\lambda f = Gf'.$ This differential equation  has a unique solution up to a constant $e^{\lambda h} \in \hol(\D)$ (see \cite{Siskakis}), where $h$ is the Koenigs function of the semigroup $(\Fi_t)_{t\geq 0}.$ Thus, $\lambda \in \sigma_p(\Delta)$ if and only if $e^{\lambda h} \in H^p,$ and in such a case $\ker(\Delta -\lambda I) = \Span\{e^{\lambda h}\}.$

The point spectrum $\sigma_p(\Delta)$ was characterized by Betsakos in \cite{Betsakos} for hyperbolic semigroups. We present a slight reformulation of his results that will be easier to handle for our purposes:

\begin{theorem}\cite{Betsakos}\label{point spectrum generador}
Let $(\Fi_t)_{t\geq 0}$ be a hyperbolic semigroup of holomorphic functions on $\D$ and assume without loss of generality that its Denjoy-Wolff point is $\tau =1.$ Let $\Omega$ be the associated Koenigs domain of the semigroup and $\gamma$ be the width of the smallest horizontal strip containing $\Omega.$ Let $(C_{\Fi_t})_{t\geq 0}$ be the induced semigroup of composition operators on $H^p$ with $1\leq p < \infty$,  and $\Delta$ its infinitesimal generator.
\begin{enumerate}
	\item [(a)] Suppose that $\Omega$ does not contain any horizontal strip. Then
	$$ \{ \lambda \in \C : - \infty < \PR(\lambda) < \frac{\pi}{p\gamma}  \} \subseteq \sigma_p(\Delta; H^p)  \subseteq  \{ \lambda \in \C : - \infty < \PR(\lambda) \leq \frac{\pi}{p\gamma}  \}.$$
	\item [(b)] Suppose that $\Omega$ contains a horizontal strip. Then there exists a horizontal strip contained in $\Omega$ with width
	$$\beta_{\textnormal{max}} = \max \{\beta > 0: \Omega \ \textnormal{contains a horizontal strip of width } \beta\}.$$ Moreover,
	$$ \{ \lambda \in \C : - \frac{\pi}{p\beta_{\textnormal{max}}} < \PR(\lambda) < \frac{\pi}{p\gamma}  \} \subseteq \sigma_p(\Delta; H^p)  \subseteq  \{ \lambda \in \C :  - \frac{\pi}{p\beta_{\textnormal{max}}} < \PR(\lambda) \leq \frac{\pi}{p\gamma}  \}.$$
\end{enumerate}
\end{theorem}

With this result at hand, we state the following proposition regarding hyperbolic semigroups. Recall that a horizontal strip is a domain defined as $$\{z+it_0 \in \C: |\Impart(z)|< t_1\}\subset \C,$$ where $t_0 \in \R$ and $t_1 > 0$.

\begin{theorem}\label{proposicion espectro}
Let $(\Fi_t)_{t\geq 0}$ be a hyperbolic semigroup of holomorphic functions on $\D$ and assume without loss of generality that is Denjoy-Wolff point is $\tau = 1.$  Let $\Omega$ be the associated Koenigs domain of the semigroup and $\gamma$ be the width of the smallest horizontal strip containing $\Omega.$ Let $(C_{\Fi_t})_{t\geq 0}$ be the induced semigroup of composition operators on $H^p$, $1\leq p < \infty$, and $\Delta$ its infinitesimal generator.
\begin{enumerate}
	\item [(a)] Assume that $\Omega$ does not contain any horizontal strip. Then,
	$$\sigma(\Delta; H^p) = \{ \lambda \in \C : -\infty < \PR(\lambda) \leq  \frac{\pi}{p\gamma} \}.$$
	\item [(b)] Assume that $\Omega$ contains a horizontal strip, and let $\beta_{\textnormal{max}}$ be as in Theorem \ref{point spectrum generador}.
	\begin{enumerate}
		\item [(i)] If $\Omega$ is a horizontal strip, then $C_{\Fi_t}$ is invertible for every $t\geq 0$ and $\beta_{\textnormal{max}}= \gamma$ and
$$\sigma(\Delta; H^p) = \{ \lambda\in\C: |\PR(\lambda)| \leq \frac{\pi}{p\gamma} \}.$$
		\item [(ii)] If $\Omega$ is not a horizontal strip, then $C_{\Fi_t}$ is not invertible for any $t> 0$.% \textcolor{red}{and there exists a sequence $(\lambda_n) \subset \sigma(\Delta; H^p)$ such that $\lim\limits_{n\rightarrow \infty} \PR(\lambda_n) = - \infty.$}
	\end{enumerate}
\end{enumerate}
\end{theorem}

\begin{proof}
We start proving $(a)$. If $\Omega$ does not contain any horizontal strip, by Theorem \ref{point spectrum generador}, it follows that
$$\{ \lambda \in \C : -\infty < \PR(\lambda) \leq  \frac{\pi}{p\gamma} \}\subseteq \sigma(\Delta; H^p).$$
For the reverse containment,  recall that the spectral radius of $C_{\Fi_t}$ in $H^p$ is given by $r(C_{\Fi_t}; H^p)= \Fi_t'(1)^{-1/p}$ (see \cite[Corollary 3.2]{Siskakis4}, for instance). Moreover,  $\Fi_t'(1) = e^{-\frac{\pi t}{\gamma}}$ (see  \cite[Theorem 2.1]{CD}, for instance), so
$$r(C_{\Fi_t}; H^p) = e^{\frac{\pi t}{\gamma p}}.$$
Upon applying Proposition \ref{spectral mapping resolvente}, it follows that for every $\lambda \in \sigma(\Delta)$ the inequality $\PR(\lambda) \leq \frac{\pi}{p\gamma}$ holds and $(a)$ is proved.
	
\smallskip

Now, we deal with $(b)$. Assume that $\Omega$ contains a horizontal strip and consider the following two cases:

\begin{enumerate}
\item [(i)] Assume $\Omega$ is a horizontal strip, and let
$$g(z) = \log\left(\frac{1+z}{1-z}\right), \qquad (z\in \D),$$
where we are considering the principal branch of the logarithm. Note that $g$ is a univalent function with $g(0)=0$ that maps the disc $\D$ onto the horizontal strip symmetric with respect to the real axis and with width $\pi.$ It is clear that there exists $\al_1,\al_2 > 0$ such that $\tilde{g}(z) := \al_1 g(z)+\al_2$ maps the unit disc onto $\Omega.$ Now, if $h$ is the Koenigs function of the semigroup, it is clear that the map
$$\Gamma : z\in \D \mapsto h^{-1}\circ \tilde{g}(z) \in \D$$
is a biholomorphic map of $\D$. Moreover, it is easy to check that $C_\Gamma C_{\Fi_t} = C_{\psi_t} C_\Gamma,$ where $\psi_t(z) = \tilde{g}^{-1}(\tilde{g}(z)+t).$ Then, the operators $C_{\Fi_t}$ and $C_{\psi_t}$ are similar for all $t\geq 0$, so it suffices to show the result for $C_{\psi_t}$. But, in this case,
$$\psi_t(z) = \frac{(e^{t/\al_1}+1)z+(e^{t/\al_1}-1)}{(e^{t/\al_1}-1)z+(e^{t/\al_1}+1)} \qquad (z\in \D),$$
which turns out to be a hyperbolic automorphism of $\D$ and, hence, $C_{\psi_t}$ is invertible.

Accordingly, $C_{\Fi_t}$ is invertible for every $t\geq 0$ and the assertion about the spectrum follows from \cite[Theorem 6]{Nordgren}, Theorem \ref{point spectrum generador} and Proposition \ref{spectral mapping resolvente}.

\smallskip

\item [(ii)] Assume $\Omega$ contains a horizontal strip but it is not one. We will argue by contradiction. Let us suppose that there exists $t_0 >0$ such that $C_{\Fi_{t_0}}$ is invertible. Then, $\Fi_{t_0}$ is an automorphism of the unit disc which is hyperbolic. Now, by \cite[Theorem 3.2]{BCD1}, every member of the semigroup is  a linear fractional map, and since they all commute, by \cite[Theorem 8]{CDMV}, it follows that every $\Fi_t$ with $t>0$ has to be a hyperbolic automorphism of the disc having the same fixed points. Finally, arguing as in the proof of $(i)$, it is easy to see that the Koenigs function of the semigroup $(\Fi_t)_{t\geq 0}$ (which can be extended to a group) maps the disc $\D$ onto a horizontal strip, which contradicts our assumption. Then, we have shown that $C_{\Fi_t}$ is not invertible for any $t > 0$, and the proof is done.
	\end{enumerate}
\end{proof}

\medskip

At this point, we can discuss the density of the eigenfunctions of infinitesimal generators of hyperbolic semigroups, which will be carried out by applying a result of \cite{BGY}:

\begin{proposition}\label{proposicion densidad}
Let $(\Fi_t)_{t\geq 0}$ be a hyperbolic semigroup of holomorphic functions in $\D$ and let $h$ and $\Omega$ denote its Koenigs function and Koenigs domain, respectively. Assume, without loss of generality, that $\tau =1.$ Let $(C_{\Fi_t})_{t\geq 0}$ be the induced semigroup of composition operators acting on $H^p$ ($1\leq p < \infty$) and $\Delta$ its infinitesimal generator. Assume that
\begin{enumerate}
	\item [(i)] Either $\Omega$ does not contain any horizontal strip, $\inte(\overline{\Omega})= \Omega$ and $\C\setminus\overline{\Omega}$ has at most two connected components,
	\item [(ii)] or $\Omega$ is a horizontal strip.
\end{enumerate}
Then, the linear manifold
$$ \Span \{e^{\lambda h} \in H^p: \lambda \in \inte(\sigma(\Delta; H^p) \}$$
is dense in $H^p$ for every $1\leq p < \infty.$
\end{proposition}

The proof is based on the ideas developed in \cite{BGY}, and we include it for the sake of completeness.

\begin{proof}
For $1\leq p<\infty$, let $H^p(\Omega)$ be the Hardy space associated to the Koenigs domain $\Omega$, that is, the Banach space consisting of holomorphic functions $f:\Omega \rightarrow \C$ such that there exists a harmonic function $u:\Omega \rightarrow \R$ such that $|f(z)|^p \leq u(z)$ for every $z \in \Omega$ (see Duren's book \cite[Chapter 10]{Duren} for more on these spaces). It turns out that the map $f \in H^p(\Omega) \mapsto f\circ h \in H^p$ is a surjective isometry, so it suffices to show that the linear manifold $\Span \{e^{\lambda z} : \lambda \in \inte(\sigma(\Delta; H^p)) \}$ is dense in $H^p(\Omega).$
	
Let $H^\infty(\Omega)$ denote the space of bounded holomorphic functions in $\Omega.$ By Theorems \ref{point spectrum generador} and \ref{proposicion espectro} it follows that $\inte(\sigma(\Delta;H^p))\subseteq \sigma_p(\Delta;H^p)$, so for every $\lambda \in \inte(\sigma(\Delta\mid_{H^p}))$ the kernel $\ker(\Delta - \lambda I)$  has dimension one and it is spanned by $e^{\lambda h}$. Observe that, by Theorem \ref{proposicion espectro} and the convexity in the positive direction of $\Omega,$  $$\{e^{\lambda z}\in H^\infty(\Omega): \lambda \in \C \}\subset \{e^{\lambda h} \in H^p: \lambda \in \inte(\sigma(\Delta; H^p)) \}$$ for all $1\leq p < \infty.$ Now,  \cite[Theorem 1.1]{BGY} yields that  the linear manifold
$$ \Span\{e^{\lambda z}\in H^\infty(\Omega): \lambda \in \C \}$$
is weak-star dense in $H^\infty(\Omega).$ As a consequence, if $1<p<\infty,$ such linear manifold is weakly dense in $H^p(\Omega),$ and by Mazur's Lemma (\cite[Corollary 3, Chapter 2]{Diestel}), it is dense on $H^p(\Omega).$ Finally, the continuity of the inclusion $H^p(\Omega) \hookrightarrow H^1(\Omega)$ and the density of $H^p(\Omega)$ in $H^1(\Omega)$ yields the density for $p=1,$ which completes the proof.
\end{proof}

The next theorem is the counterpart of Theorem \ref{main result} in the setting of hyperbolic semigroups of composition operators:

\begin{theorem}\label{teorema semigrupos}
Let $(\Fi_t)_{t\geq 0}$ be a hyperbolic semigroup of holomorphic functions in $\D$ and $\Omega$  its Koenigs domain. Assume, without loss of generality, that $\tau =1.$ Let $(C_{\Fi_t})_{t\geq 0}$ be the induced semigroup of composition operators in $X=H^p$, $1\leq p < \infty$, and $\Delta$ its infinitesimal generator of type $\omega_0\in \mathbb{R}$. Assume that
\begin{enumerate}
	\item [(i)] Either $\Omega$ does not contain any horizontal strip, $\inte(\overline{\Omega})= \Omega$ and $\C\setminus\overline{\Omega}$ has at most two connected components,
	\item [(ii)] or $\Omega$ is a horizontal strip.
\end{enumerate}
Let $\nu$ be a finite complex Borel measure supported in $[0,\infty)$ such that $\int_0^\infty e^{(w_0+\delta)t} d\nu(t) < \infty$ for some $\delta >0$ and $\EL(\nu)$ is a non-constant function in $\mathcal{E}(-\Delta)$. If $\HH_\nu$ is the subordinated operator associated to $(C_{\Fi_t})_{t\geq 0}$  the following holds:
\begin{enumerate}
	\item $\HH_\nu$ does not have the SVEP.
	\item  For every relatively open subset $U\subset \sigma(\HH_\nu),$ $\X_{\HH_\nu}(\overline{U})$ is dense in $X$.
	\item $\sigma_p(\HH_\nu^*) = \emptyset.$
	\item  $X_{\HH_\nu^*}(F) = \{0\}$ for every closed subset $F\subsetneq \sigma(\HH_\nu^*).$
	\item  $\HH_\nu^*$ has the Dunford property (C).
	\item  $\sigma_{\HH_\nu^*}(f) = \sigma(\HH_\nu^*)$ for every $f\in X^*\setminus \{0\}.$
	\item  $r_{\HH_\nu^*}(f) = r(\HH_\nu^*)$ for every $f\in X^*\setminus \{0\}.$
	\item  If $M$ is a non-trivial closed invariant subspace for $\HH_\nu^*$, then
	$ \sigma(\HH_\nu^*) \subseteq \sigma(\HH_\nu^*\mid_M)\subset \eta(\sigma(\HH_\nu^*)).$
\end{enumerate}
\end{theorem}

The proof follows as an application of Theorem \ref{main result}. Note that by Theorem \ref{proposicion espectro}, $\overline{\inte(\sigma(-\Delta))}= \sigma(-\Delta)$, so the map $$f_\Delta : z \in \inte(\sigma(-\Delta)) \mapsto e^{-zh} \in \Dom(-\Delta)$$
is well defined and for every $-z \in \inte(\sigma(\Delta))$
$$(-\Delta-zI)f_\Delta(z) = -(\Delta- (-z)I)e^{-z h} = 0,$$
since  $\inte(\sigma(\Delta))\subset \sigma_p(\Delta)$ and $\inte(\sigma(-\Delta)) = - \inte(\sigma(\Delta))$ holds. Finally, by Proposition \ref{proposicion densidad},
$$\{ f_\Delta(\lambda): \lambda \in \inte(\sigma(-\Delta)) \}$$ is dense in $H^p$, so the conditions of Theorem \ref{main result} are therefore fulfilled.

\smallskip

\smallskip

In particular, as it was discussed in Remark \ref{remark semigrupo}, it is possible to derive the following consequence for each operator in the hyperbolic composition semigroup $(C_{\Fi_t})_{t\geq 0}$:

\smallskip

\begin{corollary}\label{corolario semigrupos composicion}
Let $(\Fi_t)_{t\geq 0}$ be a hyperbolic semigroup of holomorphic functions in $\D$ and $\Omega$  its Koenigs domain. Assume, without loss of generality, that $\tau =1.$ Let $(C_{\Fi_t})_{t\geq 0}$ be the induced semigroup of composition operators in $X=H^p$, $1\leq p < \infty$, and $\Delta$ its infinitesimal generator. Assume that
\begin{enumerate}
	\item [(i)] Either $\Omega$ does not contain any horizontal strip, $\inte(\overline{\Omega})= \Omega$ and $\C\setminus\overline{\Omega}$ has at most two connected components;
	\item [(ii)] or $\Omega$ is a horizontal strip.
\end{enumerate}
For every $t>0$  the following holds:
\begin{enumerate}
	\item $C_{\Fi_t}$ does not have the SVEP.
	\item  For every relatively open subset $U\subset \sigma(C_{\Fi_t}),$ $\X_{C_{\Fi_t}}(\overline{U})$ is dense in $X$.
	\item $\sigma_p(C_{\Fi_t}^*) = \emptyset.$
	\item  $X_{C_{\Fi_t}^*}(F) = \{0\}$ for every closed subset $F\subsetneq \sigma(C_{\Fi_t}^*).$
	\item  $C_{\Fi_t}^*$ has the Dunford property (C).
	\item  $\sigma_{C_{\Fi_t}^*}(f) = \sigma(C_{\Fi_t}^*)$ for every $f\in X^*\setminus \{0\}.$
	\item  $r_{C_{\Fi_t}^*}(f) = r(C_{\Fi_t}^*)$ for every $f\in X^*\setminus \{0\}.$
	\item  If $M$ is a non-trivial closed invariant subspace for $T^*$, then
	$ \sigma(C_{\Fi_t}^*) \subseteq \sigma(C_{\Fi_t}^*\mid_M)\subset \eta(\sigma(C_{\Fi_t}^*)).$
\end{enumerate}
\end{corollary}

\begin{proof}
Assume that $\Omega$ does not contain any horizontal strip, the proof of Theorem \ref{proposicion espectro} yields that
\begin{equation}\label{spectro Fit}\sigma(C_{\Fi_t}; H^p) = \overline{\{e^{\lambda t}: \lambda \in \sigma(\Delta) \}}\end{equation} for every $t> 0$ having in mind Proposition \ref{spectral mapping resolvente} and $r(C_{\Fi_t}; H^p)= e^{\frac{\pi t}{\gamma p}}$.

On the other hand, if $\Omega$ is a horizontal strip, the semigroup is similar to a semigroup of hyperbolic automorphisms and in such a case Theorem \ref{point spectrum generador} along with \cite[Theorem 6]{Nordgren} yields \eqref{spectro Fit}.
\end{proof}

\medskip

\begin{remark}
It is important to note that properties $(1)-(8)$ were already proved for composition operators induced by hyperbolic linear fractional maps in \cite{GGM}. Corollary \ref{corolario semigrupos composicion} extends these results to every composition operator that can be embedded in a hyperbolic semigroup such that its Koenigs domain meets the geometrical conditions $(i)$ or $(ii)$.
\end{remark}

\subsection{Examples of subordinated operators}\label{subseccion} In this subsection, we provide particular instances of hyperbolic semigroups of holomorphic functions in $\D$ which yield examples of subordinated operators for which Theorem \ref{teorema semigrupos} applies. Of particular interest is the classical Ces\`aro operator.

\begin{enumerate}
\item For each $t\geq 0$, let
$$
\Fi_t(z) = e^{-t}z + (1-e^{-t}), \qquad (z\in \D).
$$
Each $\Fi_t$ is an affine map, in particular a hyperbolic non-automorphism linear fractional map in $\D$ which fixes $1$ and $\infty$.
Clearly $(\Fi_t)_{t\geq 0}$ is a semigroup of holomorphic functions in $\D$ and  its Koenigs function is given by $h(z) = -\log(1-z)$, for $z\in \D$,  where the principal branch of the logarithm is considered. The Koenigs domain $\Omega$ is contained in
$$\{ z\in \C: \PR(z) \geq -\log(2),\; -\pi/2 < \Impart(z)<\pi/2 \}$$
and the semigroup is hyperbolic since it is contained in a horizontal strip. Indeed,  $\Omega$ does not contain any horizontal strip and $\inte(\overline{\Omega})= \Omega$, so Theorem \ref{teorema semigrupos} clearly applies to the semigroup $(C_{\Fi_t})_{t\geq 0}$ acting on $H^p.$

\medskip

It is well known that $\sigma(\Delta; H^p) = \{\lambda \in \C: \PR(\lambda) \leq \frac{1}{p} \}.$ For each $\lambda \in \C$ such that $\PR(\lambda) > \frac{1}{p},$  let $\nu_\lambda$ the Borel measure
$$d\nu_\lambda(t) = e^{-\lambda t}dt$$
in  $[0,\infty).$ Clearly $\nu_{\lambda}$ satisfies the hypothesis on the measure of Theorem \ref{teorema semigrupos} and by \eqref{laplace transform} $\HH_{\nu_\lambda} = R(\lambda,\Delta).$ For every $f \in H^p$ we have
\begin{equation*}
\begin{split}
\HH_{\nu_\lambda} f(z) &= \int_{0}^\infty e^{-\lambda t} (C_{\Fi_t}f)(z) dt = \int_{0}^\infty e^{-\lambda t} f(e^{-t}z+(1-e^{-t})) dt \\
& = \frac{-1}{(z-1)^\lambda}\int_{z}^1 (s-1)^{\lambda-1}f(s)ds, \qquad (z\in \D)
\end{split}
\end{equation*}
so $\HH_{\nu_\lambda}$ is an averaging operator.  Theorem \ref{teorema semigrupos} assures that properties $(1)-(8)$ are satisfied for such operators.

\medskip

In particular, consider $1<p<\infty$ and let $\lambda =1.$ In this case, we obtain the operator
$$\HH_{\nu_1}f(z) = \frac{-1}{z-1} \int_{z}^{1} f(s)ds, \qquad (z\in \D).$$
By \cite{MMS} $-\HH_{\nu_1}$ is (similar to) the adjoint operator of the Cesàro operator $\mathcal{C}$ acting on $H^p,$ that is, the operator defined as
$$(\mathcal{C}f)(z) = \int_0^z \frac{f(s)}{1-s}ds, \qquad (z\in \D).$$
Therefore, $\mathcal{C}^*$ does not enjoy the SVEP and their glocal spectral subspaces associated to the closure of relatively open sets are dense. Moreover, the spectral subspaces associated to proper closed sets of the Cesàro operator are zero, so the Cesàro operator has the Dunford property, the local spectrum at non-zero functions coincides with the spectrum of $\mathcal{C}$, and $\sigma(\mathcal{C}\mid_M) = \sigma(\mathcal{C})$ for every non-trivial closed invariant subspace of $\mathcal{C}$.

\medskip

It is worth pointing out that the Dunford property was already known for the Cesàro operator since Miller, Miller and Smith \cite{MMS} proved that $\mathcal{C}$ has the Bishop property which implies it. Nevertheless, the characterization of the spectral subspaces and the local spectra of $\mathcal{C}$ had not been previously studied in the literature.  Next result (already mentioned in the Introduction) states the new features regarding the local spectra of $\mathcal{C}$, and  as a particular instance, yields that $\mathcal{C}$ enjoys the Dunford property $(C)$.

\medskip

\begin{theorem}
Let $1<p< \infty$ and $\mathcal{C}$ be  the Cesàro operator acting on the Hardy space $H^p.$ Then:
	\begin{enumerate}
	\item  $\sigma_{\mathcal{C}}(f) = \sigma(\mathcal{C}; H^p)$ for every $f\in H^p\setminus \{0\}.$
	\item  $r_{\mathcal{C}}(f) = r(\mathcal{C}; H^p)$ for every $f\in H^p\setminus \{0\}.$
    \item $H^p_{\mathcal{C}}(F) = \{0\}$ for every closed subset $F\subsetneq \sigma(\mathcal{C}; H^p).$
    \item  If $M$ is a non-trivial closed invariant subspace of $\mathcal{C}$, then $\sigma(\mathcal{C}\mid_M; H^p)= \sigma(\mathcal{C}; H^p).$
	\end{enumerate}
\end{theorem}

\item For each $t\geq 0$, let $\Fi_t$ be the hyperbolic automorphism of $\D$ fixing $\pm 1$:
$$ \Fi_t(z) = \frac{(e^{t}+1)z+(e^{t}-1)}{(e^{t}-1)z+(e^{t}+1)} \qquad (z\in \D).$$
Clearly $(\Fi_t)_{t\geq 0}$ forms a semigroup of holomorphic functions in $\D$ that can be extended to a group $(\Fi_t)_{t\in \R}$ with the obvious definition for $t<0.$ As it was discussed in the proof of Theorem \ref{proposicion espectro}, the Koenigs domain associated to such semigroup is a horizontal strip.  In this case, $\sigma(\Delta; H^p) = \{ \lambda \in \C: |\PR(\lambda)| \leq \frac{1}{p} \}.$ For every $\lambda\in \mathbb{C}$ with $\PR(\lambda) >1$ the resolvent operator
$$(\HH_{\nu_\lambda}f)(z) = \int_{0}^\infty e^{-\lambda t} f(\Fi_t(z))dt = \left( \frac{1+z}{1-z}\right)^\lambda \int_{z}^{1} \left( \frac{1+s}{1-s}  \right)^{-\lambda} \frac{2}{1-s^2} f(s)ds, \qquad (z\in \D),$$
is an averaging operator in $H^p$ satisfying the local spectral properties of Theorem \ref{teorema semigrupos}.

\item In general, a formula for the resolvent operators arising from semigroups of composition operators can be obtained. If $(\Fi_t)_{t\geq 0}$ is a hyperbolic semigroup of holomorphic functions in $\D$ with Denjoy-Wolff point $\tau=1 \in \T$, $h$ is the associated Koenigs function and $(C_{\Fi_t})_{t\geq 0}$ the associated semigroup of type $\omega_0\in \R$,  for every $\lambda \in \C$ such that $\PR(\lambda) > \omega_0$ the resolvent at $\lambda$ may be expressed as
$$ (\HH_{\nu_\lambda}f)(z) = e^{\lambda h(z)} \int_{z}^1 h'(s)e^{-\lambda h(s)} f(s) ds, \qquad (f\in H^p,\, z\in \D).$$
 Indeed, if the Laplace transform of the measure $d\nu_\rho(t) = \rho(t)dt$ belongs to $E(-\Delta)$, where $\rho:[0,\infty) \rightarrow \C$, it follows
\begin{equation}\label{forma general subordinado}
(\HH_{\nu_\rho}f)(z) = \int_z^1 h'(s)\rho(h(s)-h(z)) f(s) ds, \qquad (f\in H^p,\, z\in \D).
\end{equation}
Thus, if $(\Fi_t)_{t\geq 0}$ is a hyperbolic semigroup satisfying the hypothesis of Theorem \ref{teorema semigrupos}, every operator of the form \eqref{forma general subordinado} has properties $(1)-(8)$ from the aforementioned Theorem.
\end{enumerate}

\medskip

\subsection{A remark on non-hyperbolic semigroups} To end up this section, we briefly comment the local spectral picture of composition operators semigroups induced by non-hyperbolic symbols, which looks quite different.

\begin{enumerate}
\item For $t\geq 0$, let $\Fi_t$ be the holomorphic function
$$\Fi_t(z) = \frac{e^{-t}z}{(e^{-t}-1)z+1} \qquad (z\in \D).$$
Clearly $\Fi_t(0)=0$ and  $(\Fi_t)_{t\geq 0}$ is an elliptic semigroup. It was shown in \cite{GGM} that the composition operator $C_{\Fi_t}$ acting on $H^p$, $1\leq p < \infty$, satisfies the Dunford's property, so in particular it has the SVEP. Moreover, by considering $\psi_t(z) =  1-(1-z)^{e^{-t}}$ for $t\geq 0$, which also forms an elliptic semigroup, the Cesàro operator appears as the resolvent at $\lambda = 1$ (\cite[p. 19]{Siskakis}):
$$\mathcal{C} = R(1,\Delta),$$
where $\Delta$ is the infinitesimal generator of $(C_{\psi_t})_{t\geq 0}$ acting on $H^p$. In this case, the Cesàro operator appears as a subordinated operator having the Bishop's property and $\mathcal{C}^*$  not enjoying the SVEP (see \cite{MMS}). This situation is drastically different to the one presented in Theorem \ref{teorema semigrupos} for subordinated operators induced by hyperbolic semigroups.

\item For $t\geq 0,$ let $\Fi_t$ be the parabolic automorphism of $\D$ fixing $1$:
$$\Fi_t(z) = \frac{ (1-it)z+it }{-itz+1+it}, \qquad (z\in \D).$$
$(\Fi_t)_{t\geq 0}$ turns out to be a semigroup of holomorphic functions in $\D$. It is well known that $\sigma(C_{\Fi_t}; H^p)=\sigma_p(C_{\Fi_t}; H^p) = \T$ (see \cite{Nordgren}). Moreover, if $\Delta$ is the infinitesimal generator of $(C_{\Fi_t})_{t\geq 0}$ acting on $H^p$, then
$$\sigma(\Delta; H^p) = \sigma_p(\Delta; H^p) = \{i\lambda : \lambda \leq 0 \}$$
(see \cite{Siskakis}). As a particular instance, the equality $\sigma(C_{\Fi_t}; H^p)= \{ e^{\lambda t}: \lambda \in \sigma(\Delta) \}$ holds. Nevertheless, the operators $C_{\Fi_t}$ are generalized scalar (see \cite{Smith}), which in particular implies that $C_{\Fi_t}$ and $C_{\Fi_t}^*$ have the Dunford's property. This fact follows since all generalized scalar operators and their adjoint are decomposable operators, and these operators have Dunford's property (\cite[Section 1.8]{LN00}). In particular, each spectral subspace of such operators associated to proper non-empty closed sets are non-trivial. Clearly, the local spectral picture of the operators of the semigroup differs noticeably from the one for hyperbolic semigroups described in Corollary \ref{corolario semigrupos composicion}.
\end{enumerate}

\end{document}